\documentclass[a4paper,reqno]{amsart}

\usepackage{amsmath,amsthm,amssymb,amsfonts,mathrsfs,graphicx,transparent}
\usepackage{esint}
\usepackage[usenames]{color}
\usepackage[a4paper,marginpar=25mm]{geometry} 

\theoremstyle{plain}
\newtheorem{theorem}{Theorem}[section]
\newtheorem{corollary}[theorem]{Corollary}

\newtheorem{lemma}[theorem]{Lemma}

\theoremstyle{definition}

\theoremstyle{remark}
\newtheorem{remark}{Remark}[section]
\newtheorem{example}{Example}[section]
\numberwithin{equation}{section}

\begin{document}

\title[Non-planar interfaces]{An investigation of non-planar austenite-martensite interfaces}

\author{Konstantinos Koumatos}
\address{Konstantinos Koumatos: Mathematical Institute, University of Oxford, 24--29 St Giles', Oxford OX1 3LB, United Kingdom.}
\email{koumatos@maths.ox.ac.uk}

\author{John M. Ball}
\address{John M. Ball: Mathematical Institute, University of Oxford, 24--29 St Giles', Oxford OX1 3LB, United Kingdom.}
\email{ball@maths.ox.ac.uk}

\begin{abstract}
Motivated by experimental observations on CuAlNi single crystals, we present a theoretical investigation of non-planar austenite-martensite interfaces. 
Our analysis is based on the nonlinear elasticity model for martensitic transformations and we show that, under 
suitable assumptions on the lattice parameters, non-planar interfaces are possible, in particular for transitions with cubic austenite.
\vspace{4pt}

\vspace{2pt}

\end{abstract}

\maketitle

\section{Introduction}

A classical austenite-martensite interface is a plane - the habit plane - separating undistorted austenite from a simple laminate of martensite, i.e.~a 
region where the deformation gradient jumps between two constant matrices, say $A$ and $B$, on alternating bands of width $\lambda$ and 
$1-\lambda$, respectively, where $\lambda\in(0,1)$. These interfaces have been broadly studied and are well understood. On the other hand, the nonlinear 
elasticity model for martensitic transformations (see e.g.~\cite{bj87,bj92}) allows for interfaces separating undistorted austenite from more 
complicated microstructures of martensite; such interfaces are broadly referred to as non-classical.

Seiner and Landa \cite{seinercurved}, observed such non-classical interfaces in a CuAlNi shape-memory alloy, in which the martensite 
consists of two laminates - a compound and a Type-II twin\footnote{The terminology is not important for our purposes and the reader is referred to 
e.g.~\cite{seinercurved}} - crossing each other; this morphology is usually referred to as parallelogram or twin-crossing microstructure. More strikingly, 
the volume fraction, say $\Lambda$, of the compound twin varied as a function of position, resulting in a non-planar habit surface.

In \cite{icomat08}, an analysis was provided for the macroscopically homogeneous case, i.e.~when the volume fractions of the two crossing laminates remain constant. 
It was shown that the observed non-classical austenite-martensite interface is compatible in the sense that, under restrictions on the lattice 
parameters, given any compound volume fraction $\Lambda\in(0,1)$, there exist precisely two Type-II volume fractions $\lambda$ ensuring continuity of 
the overall deformation across a planar interface; see Fig.~\ref{fig:volfracs}.

\begin{figure}[ht]
	\centering
	\def\svgwidth{0.6\columnwidth}
	\begingroup
    \setlength{\unitlength}{\svgwidth}
  \begin{picture}(1,0.8)%
    \put(0,0){\includegraphics[width=\unitlength]{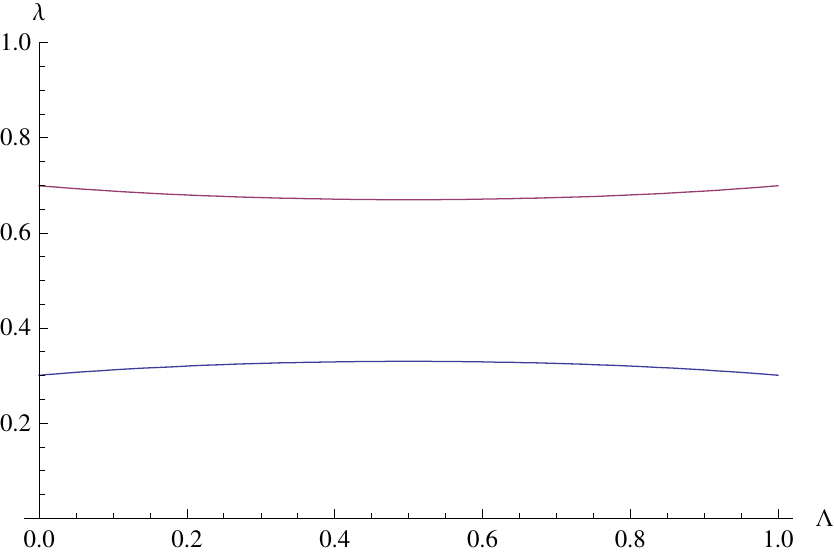}}%
    \put(0.3,0.6){\color[rgb]{0,0,0}\makebox(0,0)[lb]{\smash{$\lambda^2-\lambda=\frac{a_o+a_2(\Lambda^2-\Lambda)}{a_1+a_3(\Lambda^2-\Lambda)}$}}}%
  \end{picture}%
\endgroup
	\caption{Type-II volume fractions $\lambda$ that make the interface compatible plotted against the compound volume fraction $\Lambda\in[0,1]$; 
$a_i$, $i=0,\ldots3$ are functions of the lattice parameters.}
	\label{fig:volfracs}
\end{figure}

Equivalently, the continuity of the overall deformation across the planar interface can be expressed in terms of Hadamard's jump condition, i.e.~that 
there exist vectors $b$, $m$ and a rotation $R$ - all being functions of the volume fractions $\lambda$, $\Lambda$ - such that
\begin{equation}
\label{eq:compatibility}
 R(\lambda,\Lambda)M(\lambda,\Lambda)=\mathbf{1}+b(\lambda,\Lambda)\otimes m(\lambda,\Lambda),
\end{equation}
where $M(\lambda,\Lambda)$ denotes the macroscopic deformation gradient corresponding to the parallelogram microstructure with volume fractions $\lambda$, $\Lambda$. In particular, $m(\lambda,\Lambda)$ is the normal to the habit plane and varying the volume fraction $\Lambda$ (as in the experimental observations) forces 
the habit plane normal to vary accordingly, giving some insight into the inhomogeneous case; see Fig.~\ref{fig:normals}. However, explicit attempts to 
treat the inhomogeneous microstructure analytically proved to be either intractable, due to the algebraic complexity of the cubic-to-orthorhombic 
transition of CuAlNi, or in some cases seemingly impossible.

\begin{figure}[ht]
\centering
\includegraphics[scale=0.6]{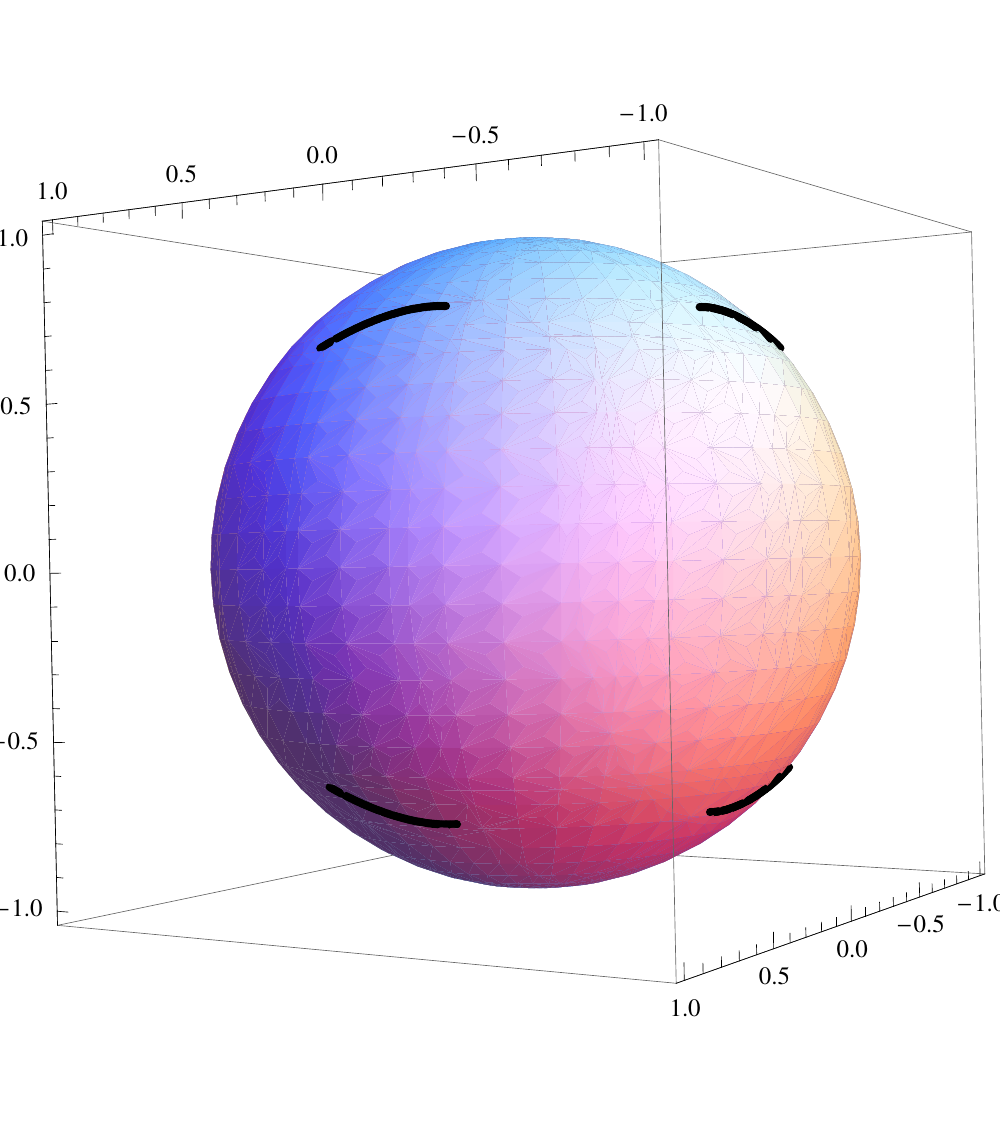} 
\caption{Plot of the unit normals $m(\lambda,\Lambda)$ in the compatibility equation (\ref{eq:compatibility}) for $(\lambda,\Lambda)$ as in 
Fig.~\ref{fig:volfracs}.}
\label{fig:normals}
\end{figure}

In this paper, a theoretical approach is followed in order to investigate the possibility of non-planar austenite-martensite interfaces. It is shown that 
such interfaces are possible within the nonlinear elasticity model for all martensitic transformations with cubic austenite. 
In Section~\ref{sec:model}, the nonlinear elasticity model is briefly introduced and non-classical interfaces are 
explained in greater depth. Our results are presented in Section~\ref{sec:results}. In particular, Theorem~\ref{lemmamodifiedjump} provides a 
modification of Hadamard's jump condition allowing for non-planar interfaces and, in Lemma~\ref{lemma:curvedinterface}, we present a general 
construction of non-planar interfaces. This construction is only possible under the assumption that the austenitic well is rank-one connected to the 
interior of the quasiconvex hull of the martensitic wells relative to a determinant constraint; see Section~\ref{sec:model} for the terminology. In Corollary~\ref{corollary:interiorpointcubicaustenite} we show that this non-trivial assumption is satisfied in particular whenever the austenite is cubic, under appropriate restrictions on the lattice parameters.

\section{Nonlinear elasticity model}
\label{sec:model}

In the nonlinear elasticity model microstructures are identified with weak$\ast$ limits of minimizing sequences (assumed bounded in 
$W^{1,\infty}(\Omega,\mathbb{R}^3)$) for a total free energy of the form:
$$I_\theta(y)=\int_{\Omega}\varphi(Dy(x),\theta)\,\mathrm{d}x.$$
We note that interfacial energy contributions are ignored, resulting in the prediction of infinitely fine microstructure. 
Above, $\Omega$ represents the reference configuration of undistorted austenite at the transformation temperature $\theta_{c}$ and $y(x)$ denotes the 
deformed position of the particle $x\in\Omega$. The free-energy function $\varphi(F,\theta)$ depends on the deformation gradient $F\in\mathbb{R}^{3\times3}$ 
and the temperature $\theta$, where $\mathbb{R}^{3\times3}$ denotes the space of real $3\times3$ matrices. By frame indifference, 
$\varphi(RF,\theta)=\varphi(F,\theta)$ for all $F$, $\theta$ and for all rotations $R$; that is for all $3\times3$ matrices in 
$${\rm SO}(3)=\left\{R:R^{T}R=\mathbf{1}, \det{R}=1\right\}.$$
Also, by material symmetry $\varphi(FQ,\theta)=\varphi(F,\theta)$ for all $Q\in\mathcal{P}^a$, where $\mathcal{P}^a\subset {\rm SO}(3)$ denotes the symmetry 
group of the austenite, e.g~for cubic austenite $\mathcal{P}^a=\mathcal{P}^{24}$ consists of the 24 rotations mapping the cube back to itself.
Without loss of generality we assume that $\min_{F} \varphi(F,\theta)=0$ and we denote by
\begin{equation}
K_{\theta}=\lbrace F:\varphi\left(F,\theta\right)=0\rbrace\nonumber
\end{equation}
the zero set of $\varphi(\cdot,\theta)$. Assuming a transformation strain $U_1(\theta)$ and using frame indifference and material symmetry, we 
suppose that
\[K_{\theta}=\left\{\begin{array}{ll}
\alpha\left(\theta\right){\rm SO}\left(3\right)&\,\theta>\theta_{c}\mbox{ (austenite)}\\
{\rm SO}\left(3\right)\cup\bigcup^{N}_{i=1}{\rm SO}\left(3\right)U_{i}\left(\theta_c\right)&\,\theta=\theta_{c}\\
\bigcup^{N}_{i=1}{\rm SO}\left(3\right)U_{i}\left(\theta\right)&\,\theta<\theta_{c}\mbox{ (martensite)},
\end{array}\right.\]
where the positive definite, symmetric matrices $U_{i}\left(\theta\right)\in\left\{R^TU_1(\theta)R:R\in\mathcal{P}^a\right\}$ correspond to the $N$ 
distinct variants of martensite and $\alpha(\theta)$ is the thermal expansion coefficient of the austenite with $\alpha(\theta_{c})=1$.

As a simple illustration of non-classical interfaces, let us restrict attention to planar ones. Then, a non-classical planar austenite-martensite interface 
$\left\{x\cdot m=k\right\}$ at the critical temperature $\theta_c$ corresponds to a choice of habit-plane normal $m$ for which there exists an energy-minimizing sequence of deformations $y^{j}$ 
such that, as $j\rightarrow\infty$,
\begin{eqnarray}
\label{convinmeasure1}
Dy^j&\rightarrow&{\rm SO}(3)\,\,\mbox{in measure for $x\cdot m<k$,}\\
\label{convinmeasure2}
Dy^j&\rightarrow&K:=\bigcup^{N}_{i=1}{\rm SO}(3)U_i(\theta_c)\,\,\mbox{in measure for $x\cdot m>k$,}
\end{eqnarray}
i.e.~$y^j$ corresponds to a pure phase of austenite and a general zero-energy microstructure of martensite on either side of the interface 
$\left\{x\cdot m=k\right\}$. Above, the convergence in measure of the sequence $Dy^j$ to a compact set $L$ means that for all open neighbourhoods $U$ of 
$L$ in $\mathbb{R}^{3\times3}$,
$$\lim_{j\rightarrow\infty}\mathrm{meas}\left\{x\in\Omega : Dy^j(x)\notin U\right\}=0.$$
In fact, this is equivalent to the Young measure $\nu=(\nu_x)_{x\in\Omega}$ generated by $Dy^j$ being supported in $L$ (see e.g.~\cite{mullernotes} for details).

Without loss of generality, (\ref{convinmeasure1}) reduces to $Dy^{j}\left(x\right)\longrightarrow\mathbf{1}$ a.e.~in $\Omega$. As for the martensitic 
region, $x\cdot m>k$, let us assume for simplicity that the martensitic microstructure is homogeneous; that is, for $x\cdot m>k$, the macroscopic 
deformation gradient $F=Dy(x)$, i.e.~the weak$\ast$ limit in $L^{\infty}(\Omega,\mathbb{R}^{3\times3})$ of $Dy^j$ satisfying (\ref{convinmeasure2}), is 
independent of $x$. We note that such matrices $F$ are precisely the elements of the \textit{quasiconvex hull} of the set $K$, denoted by $K^{qc}$. (In general, for a compact set $L\subset\mathbb{R}^{3\times3}$, say, its quasiconvex hull $L^{qc}$ is given by the set of matrices $F$ such that there exists a 
sequence of deformations $z^{j}$ uniformly bounded in $W^{1,\infty}(\Omega,\mathbb{R}^3)$ with $z^{j}\stackrel{\ast}{\rightharpoonup}Fx$ and 
$Dz^{j}(x)\rightarrow L$ in measure\footnote{The quasiconvex hull of a compact set $L$ of matrices can be equivalently defined in various ways (see e.g.~\cite{mullernotes}) but 
the above definition will suffice for our purposes.}.)

To make the overall deformation continuous across the planar interface one needs to satisfy the Hadamard jump condition as in (\ref{eq:compatibility}). 
Thus, for planar interfaces between austenite and a homogeneous microstructure of martensite, accounting for non-classical interfaces becomes equivalent to 
establishing rank-one connections between the austenite and the set $K^{qc}$, i.e.~finding vectors $b,\,m$ such that
\begin{equation}
\label{eq:twinningkqc}
\mathbf{1}+b\otimes m\in K^{qc},
\end{equation}
where, by frame indifference, we have chosen the identity matrix $\mathbf{1}$ to represent the austenite energy well. However, in this context, the only known 
characterization of a quasiconvex hull is for two martensitic wells, that is when $K={\rm SO}(3)U_{1}\cup {\rm SO}(3)U_{2}$, in which case any $F\in K^{qc}$ can be 
obtained as the macroscopic deformation gradient of a double laminate (see \cite{bj92}). Using this characterization Ball \& Carstensen were able to analyze planar non-classical interfaces for cubic-to-tetragonal transformations and the reader is referred to \cite{jmbcc97} for details.

In the present paper, we wish to work in a more general setting where we allow the martensitic microstructure to depend on the position vector and the 
interface to be represented by a general ($C^1$) surface $\Gamma$. In this case, one still needs to require that (\ref{convinmeasure1}) and 
(\ref{convinmeasure2}) hold on either side of $\Gamma$ but, since we allow the macroscopic deformation gradient $Dy$ of the martensitic microstructure 
to depend on $x\in\Omega$, we require that $Dy(x)\in K^{qc}$ a.e.~in the martensitic region. However, the compatibility condition across the interface 
$\Gamma$ no longer suffices and needs to be generalized. An appropriate generalization is provided in the subsequent section along with statements and 
proofs of our results.

\section{Construction of non-planar interfaces}
\label{sec:results}

The first step to constructing a non-planar interface between austenite and an inhomogeneous microstructure of martensite is an appropriate 
generalization of Hadamard's jump condition; this is Theorem~\ref{lemmamodifiedjump}. However, before stating and proving the generalized jump 
condition, let us first clarify notation and terminology, as well as prove an auxiliary lemma (Lemma~\ref{piecewisepath}) used in the proof of 
Theorem~\ref{lemmamodifiedjump}.\vspace{0.2cm}

\noindent\textbf{Notation.}
\begin{itemize}
\item The term \textit{domain} is reserved for an open and connected set in $\mathbb{R}^d$ throughout this paper.
\item A function $y:\Omega\rightarrow\mathbb{R}$ belongs to the space $C^{1}\left(\overline{\Omega}\right)$ if $y\in C^{1}\left(\Omega\right)$ 
and $y$ can be extended to a continuously differentiable function on an open set containing $\overline{\Omega}$. The space $C^{1}\left(\overline{\Omega},\mathbb{R}^{d}\right)$ consists of those maps $y=\left(y_{1},\ldots ,y_{d}\right):\Omega\rightarrow\mathbb{R}^{d}$ such that $y_{i}\in C^{1}\left(\overline{\Omega}\right)$ for all $i=1,\ldots ,d$.
\item Let $k>0$. A domain $\Omega$ is of class $C^{k}$ if for each $\xi\in\partial\Omega$ there exist $r>0$, a Cartesian coordinate 
system in $B(\xi,r)$, with coordinates $(\bar{x},x_d)$ where $\bar{x}=(x_1,\ldots,\,x_{d-1})$, and a $C^{k}$ function 
$g:\mathbb{R}^{d-1}\rightarrow\mathbb{R}$ such that
\begin{eqnarray*}
\Omega\cap\,B(\xi ,r)&=&\left\{x\in\,B(\xi ,r)\,:\,x_d>g(\bar{x})\right\},\\
\partial\Omega\cap\,B(\xi ,r)&=&\left\{x\in\,B(\xi ,r)\,:\,x_d=g(\bar{x})\right\}.
\end{eqnarray*}
\item A $\left(d-1\right)$-surface $\Gamma$ is a relatively compact manifold of dimension $d-1$ embedded in $\mathbb{R}^{d}$ such that $\Gamma={\rm int}\Gamma$, in the sense that every point of $\Gamma$ has an open neighbourhood in $\Gamma$ homeomorphic to a ball in $\mathbb{R}^{d-1}$. We say that a $(d-1)$-surface $\Gamma$ is of class $C^{k}$ if for each $\xi\in\Gamma$ there exist $r>0$, a 
Cartesian coordinate system in $B(\xi ,r)$, with coordinates $(\bar{x},x_d)$, and a $C^{k}$ function $g:\mathbb{R}^{d-1}\rightarrow\mathbb{R}$ such that
\[\Gamma\cap\,B(\xi ,r)=\left\{x\in\,B(\xi ,r)\,:\,x_d=g(\bar{x})\right\}.\]
We note that the reference to the dimension may be dropped when this is obvious. 
\end{itemize}

\begin{lemma}
Suppose that $\Gamma\subset\mathbb{R}^{d}$ is a connected $\left(d-1\right)$-surface which is of class $C^{1}$ and let $x^{0}\in\Gamma$. Then, for all 
$x\in\Gamma$ there exists a continuous, piecewise continuously differentiable path $\gamma:\left[0,1\right]\rightarrow\Gamma$ such that $\gamma\left(0\right)=x^{0}$ 
and $\gamma\left(1\right)=x$.
\label{piecewisepath}
\end{lemma}

\begin{proof}
Let $x^0$, $x\in\Gamma$; since $\Gamma$ is a connected manifold, it is also path-connected (see e.g.~\cite{lee_manifolds}) and there exists a continuous path connecting $x^0$ and $x$. Note that the set of points on the path is compact since it is closed and contained in a relatively compact set. Also, $\Gamma$ being $C^1$, we can cover the path with balls $B(x,r_x )$ centred at $x$ of radius $r_x$ such that $\Gamma\cap B(x,r_x)$ is the graph of a $C^1$ function $g_x:\mathbb{R}^{d-1}\rightarrow\mathbb{R}$.

Extracting a finite subcover, we may assume that there are points $x^0,\ldots, x^N=x\in\Gamma$ such that the balls $B(x^j,r_j)$, $j=0,1,\ldots,\,N$, cover the path and that there exist local coordinate 
systems, say $(\bar{x}^{c_j},x^{c_j}_d)$ with $\bar{x}^{c_j}=(x^{c_j}_1,\ldots,\,x^{c_j}_{d-1})$, and $C^1$ functions $g_j:\mathbb{R}^{d-1}\rightarrow\mathbb{R}$ such 
that
\[\Gamma\cap\,B(x^j,r_j)=\left\{x\in\,B(x^j,r_j)\,:\,x^{c_j}_d=g_j(\bar{x}^{c_j})\right\}.\]
Above, the superscript $c_j$ denotes the coordinate system in the ball $B(x^j,r_j)$. Trivially, we may also assume that for $j=0,1,\ldots,\,N-1$, there exists some
$y^{j+1}\in\,B(x^j,r_j)\cap\,B(x^{j+1},r_{j+1})$.

For $j=0,\ldots,\,N$, we may define $N$ continuously 
differentiable paths $\gamma_j$ in each $B(x^j,r_j)$ such that $\gamma_j(0)=y^{j}$ and $\gamma_j(1)=y^{j+1}$, where we make the identification 
$x^{0}=y^{0}$ and $x^N=y^{N+1}$, by
\[\gamma_j(t)=\left(\bar{y}^{j,c_j}+t\left(\bar{y}^{j+1,c_j}-\bar{y}^{j,c_j}\right),g\left(\bar{y}^{j,c_j}+t\left(\bar{y}^{j+1,c_j}-
\bar{y}^{j,c_j}\right)\right)\right),\]
where the superscript $j,c_j$ denotes the point $y^j$ expressed in the coordinate system $c_j$ in $B(x^j,r_j)$. Then, the composition of
$\gamma_0,\ldots,\,\gamma_{N}$ gives a continuous, piecewise continuously differentiable path $\gamma:\left[0,1\right]\rightarrow\Gamma$ such that $\gamma(0)=x^0$ and $\gamma(1)=x$.
\end{proof}

We may now state and prove the generalized jump condition:

\begin{theorem}
\label{lemmamodifiedjump}
Let $d\geq 2$. Let $\Omega\subset\mathbb{R}^{d}$ be a bounded $C^{1}$ domain and suppose that $\Omega=\Omega^{+}\cup\Gamma\cup\Omega^{-}$ where 
$\Omega^{+},\,\Omega^{-}$ are disjoint open sets and $\Gamma=\Omega\cap\partial\Omega^{+}=\Omega\cap\partial\Omega^{-}$ is a $\left(d-1\right)$-surface 
of class $C^{1}$. Further, assume that $y^{\pm}\in C^{1}\left(\overline{\Omega^{\pm}},\mathbb{R}^{d}\right)$ and let 
$n\in C\left(\Gamma,\mathbb{R}^{d}\right)$ be the outward unit normal to $\Gamma$ with respect to $\Omega^{+}$. If there exists a map 
$z\in W^{1,\infty}\left(\Omega,\mathbb{R}^{d}\right)$ such that
\begin{equation}\label{eq:map_z}
Dz\left(x\right)=\left\lbrace\begin{array}{rcl}
Dy^{+}\left(x\right),& &x\in\Omega^{+}\\
Dy^{-}\left(x\right),& &x\in\Omega^{-},
\end{array}\right.
\end{equation}
then, for some $a\in C\left(\Gamma,\mathbb{R}^{d}\right)$ and all $x\in\Gamma$,
\begin{equation}
Dy^{+}\left(x\right)=Dy^{-}\left(x\right)+a\left(x\right)\otimes n\left(x\right),
\label{eq:jumpcondition}
\end{equation}
Conversely, suppose that $\Gamma$ is connected or $\Omega$ is simply connected. If  \eqref{eq:jumpcondition} holds then there exists a map $z\in W^{1,\infty}\left(\Omega,\mathbb{R}^{d}\right)$ satisfying \eqref{eq:map_z}. 
\end{theorem}

\begin{remark}
Note that, under the hypotheses of Theorem~\ref{lemmamodifiedjump}, $\Omega^{+}$ and $\Omega^{-}$ locally lie on either side of 
the surface $\Gamma$. Also, for the construction of a non-planar austenite-martensite interface we are interested in the special case where, say, 
$Dy^{-}=\mathbf{1}$ represents the pure phase of austenite, whereas $Dy^{+}$ represents the macroscopic deformation gradient corresponding to a microstructure of martensite, 
i.e.~$Dy^{+}\left(x\right)\in K^{qc}$ a.e.~in $\Omega^{+}$ for some set $K$ of martensitic wells; see Fig.~\ref{fig:modifiedhadamard}.\end{remark}

\begin{figure}[ht]
	\centering
	\def\svgwidth{0.5\columnwidth}
	\begingroup
    \setlength{\unitlength}{\svgwidth}
  \begin{picture}(1,0.8)%
    \put(0,0){\includegraphics[width=\unitlength]{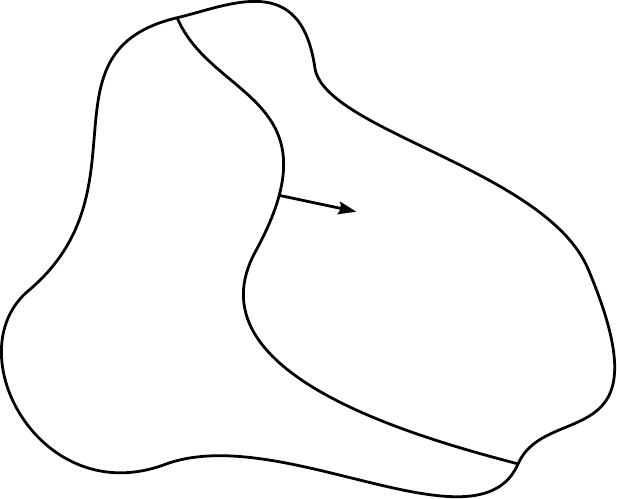}}%
    \put(0.25,0.65){\color[rgb]{0,0,0}\makebox(0,0)[lb]{\smash{$\Omega^{+}$}}}%
    \put(0.43,0.65){\color[rgb]{0,0,0}\makebox(0,0)[lb]{\smash{$\Omega^{-}$}}}%
    \put(0.8,0.1){\color[rgb]{0,0,0}\makebox(0,0)[lb]{\smash{$\Gamma$}}}%
    \put(0.55,0.4){\color[rgb]{0,0,0}\makebox(0,0)[lb]{\smash{$n\left(x\right)$}}}%
    \put(0.1,0.3){\color[rgb]{0,0,0}\makebox(0,0)[lb]{\smash{$Dy\left(x\right)$}}}%
    \put(0.8,0.3){\color[rgb]{0,0,0}\makebox(0,0)[lb]{\smash{$\mathbf{1}$}}}%
  \end{picture}%
\endgroup
	\caption{Schematic depiction of a deformation gradient $Dz$ taking the values $Dy$ and $\mathbf{1}$ on either side of a non-planar interface 
$\Gamma$; for an austenite-martensite interface the macroscopic deformation gradient $Dy\in\,K^{qc}$ a.e.~for some appropriate set $K$ of martensitic wells. The deformation $z$ remains 
continuous across $\Gamma$ provided that $Dy\left(x\right)=\mathbf{1}+a\left(x\right)\otimes n\left(x\right)$ for all $x\in\Gamma$.}
	\label{fig:modifiedhadamard}
\end{figure}
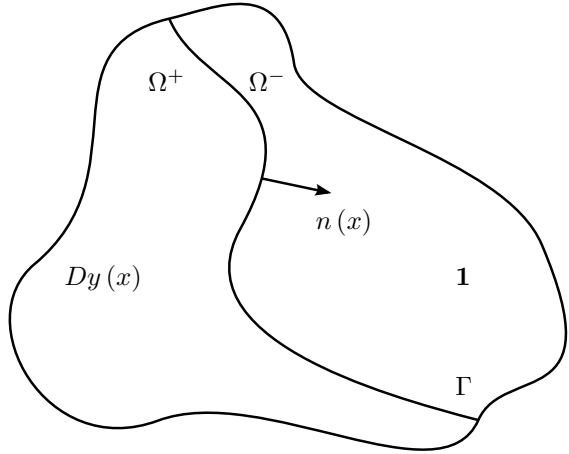

\begin{proof}
The necessity of (\ref{eq:jumpcondition}) for each $x\in \Gamma$ follows from the classical Hadamard jump condition for continuous piecewise $C^1$ maps, which can be proved by blowing-up about $x$ to reduce the case to that for a continuous piecewise affine map (see e.g.~\cite{jmbcc_inprep} for proofs of much more general statements), while the continuity of $a(\cdot)$ follows from that of $Dy^\pm$.

To prove sufficiency, assume first that $\Gamma$ is connected. Let $x_{0}\in\Gamma$ fixed and, by adding an appropriate constant, assume that $y^{+}(x_{0})=y^{-}(x_{0})=0$. It is enough to 
show that $y^{+}\left(x\right)=y^{-}\left(x\right)$, for all $x\in\Gamma$, as the map
\begin{equation}
z\left(x\right)=\left\lbrace\begin{array}{rcl}
y^{+}\left(x\right),& &x\in\Omega^{+}\\
y^{-}\left(x\right),& &x\in\Omega^{-}
\end{array}\right.\nonumber
\end{equation}
is then continuous across $\Gamma$ and $Dz$ has the required form.

Let $x\in\Gamma$; By Lemma~\ref{piecewisepath}, we can define a piecewise continuously differentiable path 
$\gamma:\left[0,1\right]\rightarrow\Gamma$ such that $\gamma(0)=x_{0}$ and $\gamma(1)=x$. Then,
\begin{equation}
y^{+}(x)-y^{-}(x)=\int^{1}_{0}\left[Dy^{+}(\gamma(t))-Dy^{-}(\gamma(t))\right]\dot{\gamma}(t)\,dt.\nonumber
\end{equation}
Note that for a.e.~$t\in\left[0,1\right]$, $\dot{\gamma}(t)$ is tangential to the path and, in particular, perpendicular to the normal at the point 
$\gamma(t)\in\Gamma$, i.e.~$\dot{\gamma}(t)\cdot n(\gamma(t))=0$. On the other hand, we know that for all $x\in\Gamma$, 
$Dy^{+}(x)-Dy^{-}(x)=a(x)\otimes n(x)$ and therefore
\begin{equation}
y^{+}(x)-y^{-}(x)=\int^{1}_{0}\left[n(\gamma(t))\cdot\dot{\gamma}(t)\right]a(\gamma(t))\,dt=0.\nonumber
\end{equation}

On the other hand, assume that $\Omega$ is simply connected. Condition (\ref{eq:jumpcondition}) now ensures that the map 
$F\in L^{\infty}(\Omega, M^{d\times d})$ defined by
\[F\left(x\right)=\left\lbrace\begin{array}{rcl}
Dy^{+}\left(x\right),& &x\in\Omega^{+}\\
Dy^{-}\left(x\right),& &x\in\Omega^{-}
\end{array}\right.\]
is curl-free (in the distributional sense) and Theorem 1 in~\cite{ciarletjr} shows that this is equivalent to the existence of a distribution with a 
distributional derivative given by $F$. Then, Maz'ya in Section 1.1.11, \cite{mazya} shows that a distribution whose derivatives of order $k$ belong 
to an $L^{p}$ space must itself be a function and an element of the Sobolev space $W^{k,p}$; sufficiency follows.
\end{proof}

\begin{remark}
(i) Note that if $\Gamma$ is disconnected and $\Omega$ is not simply connected, the result is in general false. As an example, consider $\Omega=A\times (0,1)$ where $A$ is the annulus
\[\{(x_1,x_2): 1<x_1^2+x_2^2<4\}.\]
Let $\Omega^\pm=\{(x_1,x_2,x_3): \pm x_1<0,\,x_3\in (0,1)\}$ and $\Gamma = \Gamma_1\cup\Gamma_2$ where
\begin{align*}
\Gamma_1 &=\{(x_1,x_2,x_3): x_1=0, x_2>1, x_3\in (0,1)\}\\
\Gamma_2 &=\{(x_1,x_2,x_3): x_1=0, x_2<-1, x_3\in (0,1)\}.
\end{align*}
A planar section perpendicular to the $x_3$ axis is depicted in Fig.~\ref{fig:pathological} below.
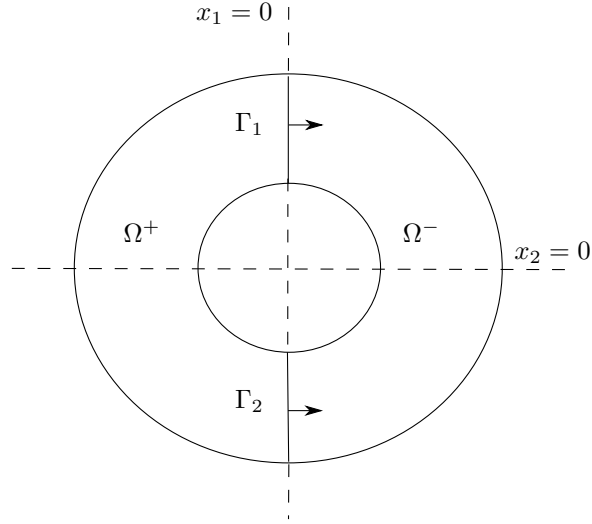
\begin{figure}[ht]
	\centering
	\def\svgwidth{0.5\columnwidth}
	\begingroup
    \setlength{\unitlength}{\svgwidth}
  \begin{picture}(1.0,0.9)%
    \put(0,0){\includegraphics[width=\unitlength]{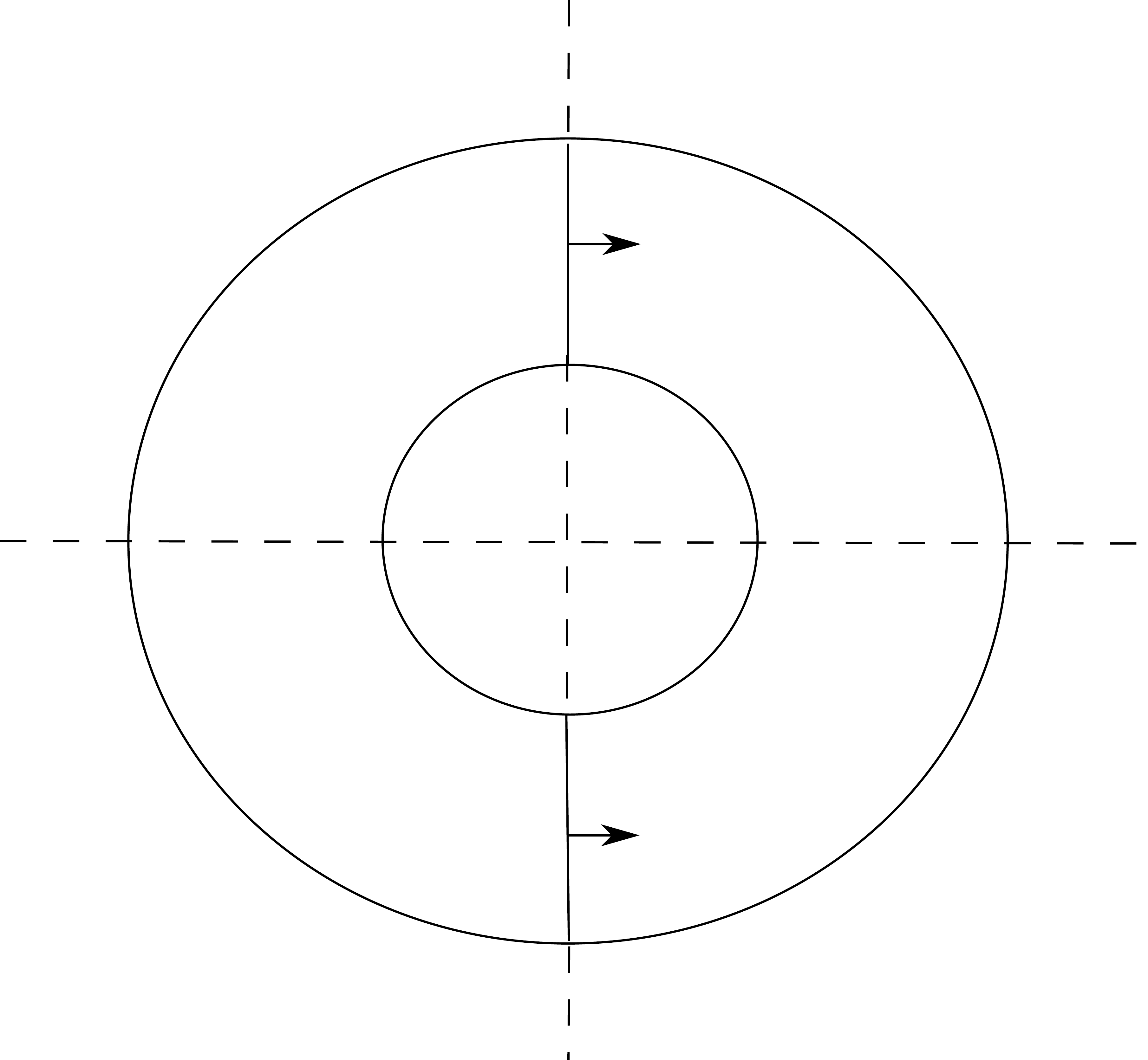}}%
    \put(0.7,0.5){\color[rgb]{0,0,0}\makebox(0,0)[lb]{\smash{$\Omega^{-}$}}}%
    \put(0.2,0.5){\color[rgb]{0,0,0}\makebox(0,0)[lb]{\smash{$\Omega^{+}$}}}%
    \put(0.4,0.2){\color[rgb]{0,0,0}\makebox(0,0)[lb]{\smash{$\Gamma_{2}$}}}%
    \put(0.4,0.7){\color[rgb]{0,0,0}\makebox(0,0)[lb]{\smash{$\Gamma_{1}$}}}%
    \put(0.33,0.9){\color[rgb]{0,0,0}\makebox(0,0)[lb]{\smash{$x_{1}=0$}}}%
    \put(0.9,0.47){\color[rgb]{0,0,0}\makebox(0,0)[lb]{\smash{$x_{2}=0$}}}%
  \end{picture}%
\endgroup
	\caption{The manifold $\Gamma=\Gamma_{1}\cup\Gamma_{2}$ is disconnected, $\Omega$ is not simply connected, and the conclusion of Theorem 3.1 no longer holds.}
	\label{fig:pathological}
\end{figure}
Let $y^{+}(x)=0$ in $\Omega^+$, $y^-(x)=0$ for $x\in\Omega^-\cap\{x_2>1\}$, $y^-(x)=ax_1+e_2$ for $x\in\Omega^-\cap\{x_2<-1\}$ and interpolate smoothly for $x_2\in [-1,1]$, where $a$ is a non-zero vector and $x_1=x\cdot e_1$, $e_{i}$ being the standard basis of $\mathbb{R}^3$; trivially, the compatibility condition (\ref{eq:jumpcondition}) is satisfied across $\Gamma$, with $Dy^+-Dy^-=0$ on $\Gamma_1$ and $Dy^+-Dy^-=a\otimes e_1$ on $\Gamma_2$. Next suppose that there exists $z\in W^{1,\infty}(\Omega,\mathbb{R}^3)$ such that
\begin{equation*}
Dz(x)=\left\{\begin{array}{ll}Dy^+(x),&x\in\Omega^+ \\
Dy^-(x), &x\in\Omega^- \end{array}\right.=\left\{\begin{array}{ll}0,&x\in\Omega^+ \\
Dy^-(x), &x\in\Omega^- .\end{array}\right.
\end{equation*}
Then, since $\Omega^\pm$ are connected, $z(x)=c$ in $\Omega^+$ for some constant $c$, whereas, in $\Omega^-$, $z(x)=y^-(x)+d$ for some other constant $d$. But continuity across $\Gamma_1$ requires that $d=c$ and hence on $\Gamma_2$ that $c=c+e_2$ - a contradiction.

(ii) For general results concerning the relationship between the gradients of a Lipschitz mapping on either side of an interface see Ball \& Carstensen \cite{jmbcc_inprep}, Iwaniec et al. \cite{iwaniec2002failure}. 
\end{remark}

Next we present a method for constructing non-planar interfaces at zero stress that is applicable to any set of martensitic wells $K$, provided that there exists a rank-one connection between ${\rm SO}\left(d\right)$, the 
austenitic well, and the relative interior of the quasiconvex hull of $K$, $\mathrm{rint}\,K^{qc}$. Here, the interior of $K^{qc}$ is taken relative to 
the set $\mathcal{D}:=\lbrace A\in\mathbb{R}^{d\times d}:\det A=\Delta\rbrace$, where $\Delta$ denotes the determinant of the martensitic variants.

\begin{lemma} Let $K\subset\mathcal{D}$ be a compact set such that $\mathrm{rint}\,K^{qc}\neq\emptyset$. Further, assume that there exist $\epsilon>0$ 
and nonzero vectors $a,\,n\in\mathbb{R}^{d}$, $\vert n\vert =1$, such that ${\bf 1}+a\otimes n\in\mathcal D$ and
\begin{equation}
\label{eq:conditionsofconstruction1}
B\left({\bf 1}+a\otimes n,\epsilon\right)\cap\mathcal D\subset K^{qc}.
\end{equation}
Then, for some open ball $\Omega=B(0,r)\subset\mathbb{R}^{d}$ and a non-planar $(d-1)$-surface $\Gamma$, as in 
Theorem~\ref{lemmamodifiedjump}, there exists a deformation $z\in W^{1,\infty}\left(\Omega,\mathbb{R}^{d}\right)$ such that
\begin{equation}
Dz(x)=\left\lbrace\begin{array}{rcl}
Dy(x)&,\quad &x\in\Omega^{+}\\
\mathbf{1}&,\quad &x\in\Omega^{-}
\end{array}\right.\nonumber
\end{equation}
with $y\in C^{1}\left(\overline{\Omega^{+}},\mathbb{R}^{d}\right)$ and $Dy(x)\in K^{qc}$ for all $x\in\overline{\Omega^{+}}$. That is, there exists a 
microstructure of martensite represented by $Dy$ which borders compatibly with a pure phase of austenite along the non-planar interface $\Gamma$.
\label{lemma:curvedinterface}
\end{lemma}

\begin{proof}
Let $f\in C^{1}\left(B(0,1)\right)$ satisfy the following properties:
\begin{itemize}
\item $f\left(0\right)=0$ and $\nabla f\left(0\right) =n$;
\item $\|\nabla f-n\|_{\infty}<\epsilon/\vert a\vert$;
\item $a\cdot\nabla f(x)=a\cdot n$ for all $x\in B(0,1)$;
\item for all sufficiently small $r\in (0,1)$, $\nabla f/\vert\nabla f\vert$ is not constant on $\{x\in B(0,r): f(x)=0\}$.
\end{itemize}
Choose an orthonormal system of coordinates $(\bar x,x_d)$ with origin at $0$ and $e_d=n$. Consider the map $g:B(0,1)\rightarrow{\mathbb R}^d$ defined by $g(x)=(\bar x,f(x))$. Since $\nabla g(0)={\bf 1}\neq 0$ it follows from the inverse function theorem that for $r>0$ sufficiently small and some neighbourhood $V$ of $0\in{\mathbb R}^d$ the map $g:B(0,r)\rightarrow V$ is invertible with $C^1$ inverse $g^{-1}$, and that 
$$\{x\in B(0,r):f(x)=0\}=\{x\in B(0,r): x_d=h(\bar x)\},$$
where $h(\bar x)=g^{-1}(\bar x,0)\cdot n$. Let $\Omega=B(0,r)$, $\Omega^\pm=\{x\in \Omega:\pm f(x)>0\}$ and $\Gamma=\{x\in\Omega: f(x)=0\}$. Note that $\Gamma$ is a $\left(d-1\right)$-surface of class $C^{1}$ and that the unit normal $n(x)$ to a point $x\in\Gamma$ is given by $n(x)=\nabla f(x)/\vert\nabla f(x)\vert$. Also since on $\Gamma$, $\nabla f/\vert\nabla f\vert$ cannot be constant, $\Gamma$ defines a non-planar interface.

We now construct the appropriate deformation. Define $y^-:\Omega^-\rightarrow {\mathbb R}^d$  by $y^-(x)=x$ and $y^+:\Omega^+\rightarrow{\mathbb R}^d$ by
\begin{equation}
y^+\left(x\right)=x+af\left(x\right).
\end{equation}
As $f\in C^{1}\left(B(0,1)\right)$ and $\Omega\subset\subset B(0,1)$, it follows that $y^+\in C^{1}\left(\overline{\Omega^{+}},\mathbb{R}^{d}\right)$ and 
$$Dy^+(x)=\mathbf{1}+a\otimes\nabla f(x).$$ In particular, $Dy^+(0)=\mathbf{1}+a\otimes n\in\mathrm{rint}\,K^{qc}$ and, for all other $x\in\Gamma$, $Dy^+(x)={\bf 1}+a(x)\otimes n(x)$, where $a(x)=|\nabla f(x)|a$. Also
$$\det Dy^+(x)=1 +a\cdot \nabla f(x)=1+a\cdot n,$$
so that $Dy^+(x)\in\mathcal D$ for all $x\in \Omega^+$. Finally, by 
(\ref{eq:conditionsofconstruction1}) and our assumption that $\|\nabla f-n\|_{\infty}<\epsilon/\vert a\vert$,
\[Dy^+(x)\in B(\mathbf{1}+a\otimes n,\epsilon)\cap\mathcal{D}\subset\,K^{qc}.\]
Theorem~\ref{lemmamodifiedjump} now applies and the proof is complete.
\end{proof}

Depending on the choice of function $f$ defining the implicit surface $\Gamma$, one can obtain a variety of interfaces.

\begin{example}
Let $h\in C^{1}(\mathbb{R})$ be such that $\dot{h}$ is not constant, $h(0)=\dot{h}(0)=0$ and $$\|\dot{h}\|_{\infty} 
<\frac{\epsilon}{\vert a\vert ^{2}}.$$
Assume $a,n\in{\mathbb R}^3$ are nonparallel and define $f\in C^{1}(\mathbb{R}^{3})$ by
\begin{equation}
f(x)=x\cdot n+h(x\cdot(a\wedge n)),\nonumber
\end{equation}
where $a\wedge n$ denotes the vector product of $a$ and $n$. If $a$ and $n$ are parallel, we can proceed similarly replacing $a\wedge n$ by any vector perpendicular to them. Then, $f(0)=0$ and
\begin{equation}
\nabla f(x)=n+\dot{h}(x\cdot a\wedge n)a\wedge n,\nonumber
\end{equation}
so that $\nabla f(0)=n$. Moreover, $a\cdot\nabla f(x)=a\cdot n$.
By the orthogonality of the vectors $n$ and $a\wedge n$,  $\Gamma:=\{f^{-1}(0)\}=\left\{(-h(x_{a\wedge n}),x_{a\wedge n},x_{m})\,:\,(x_{a\wedge n},x_{m})\in\mathbb{R}^{2}\right\}$ in some appropriate 
coordinate system where $x_n$, $x_{a\wedge n}$ and $x_m$ are coordinates in the directions $n$, $a\wedge n$ and some vector $m\in\mathbb{R}^3$ perpendicular to both $n$ and 
$a\wedge n$, respectively. This defines a $C^{1}$ surface which extends indefinitely in the direction of $a\wedge n$ and $m$ and 
$\Omega\subset\mathbb{R}^{3}$ can be chosen to be any domain intersecting $\Gamma$ and not just the possibly small neighbourhood of Lemma~\ref{lemma:curvedinterface}. Also, $\Gamma$ is non-planar as otherwise $\dot{h}(x_{a\wedge n})=\dot{h}(y_{a\wedge n})$ for all $x$, $y\in\Gamma$ which is impossible since $\dot{h}$ is not constant.

Moreover, note that $\nabla f(x)$ only changes along the direction $a\wedge n$ and the vector $m$ is always 
tangential to the surface. Then, the two-dimensional cross-sections of $\Gamma$ with planes parallel to the one spanned by the vectors $n$ and 
$a\wedge n$ are the same so that the cross-section with the plane $x\cdot m=k$ is parametrized by $r(t)=(-h(t),t,k)$; see Fig.~\ref{fig:curvedmathematica} for an example.

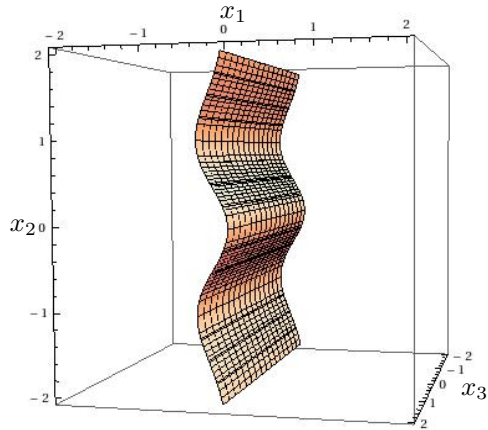
\begin{figure}[ht]
	\centering
	\def\svgwidth{0.4\columnwidth}
	\begingroup
    \setlength{\unitlength}{\svgwidth}
  \begin{picture}(1,1.1)%
    \put(0,0){\includegraphics[width=\unitlength]{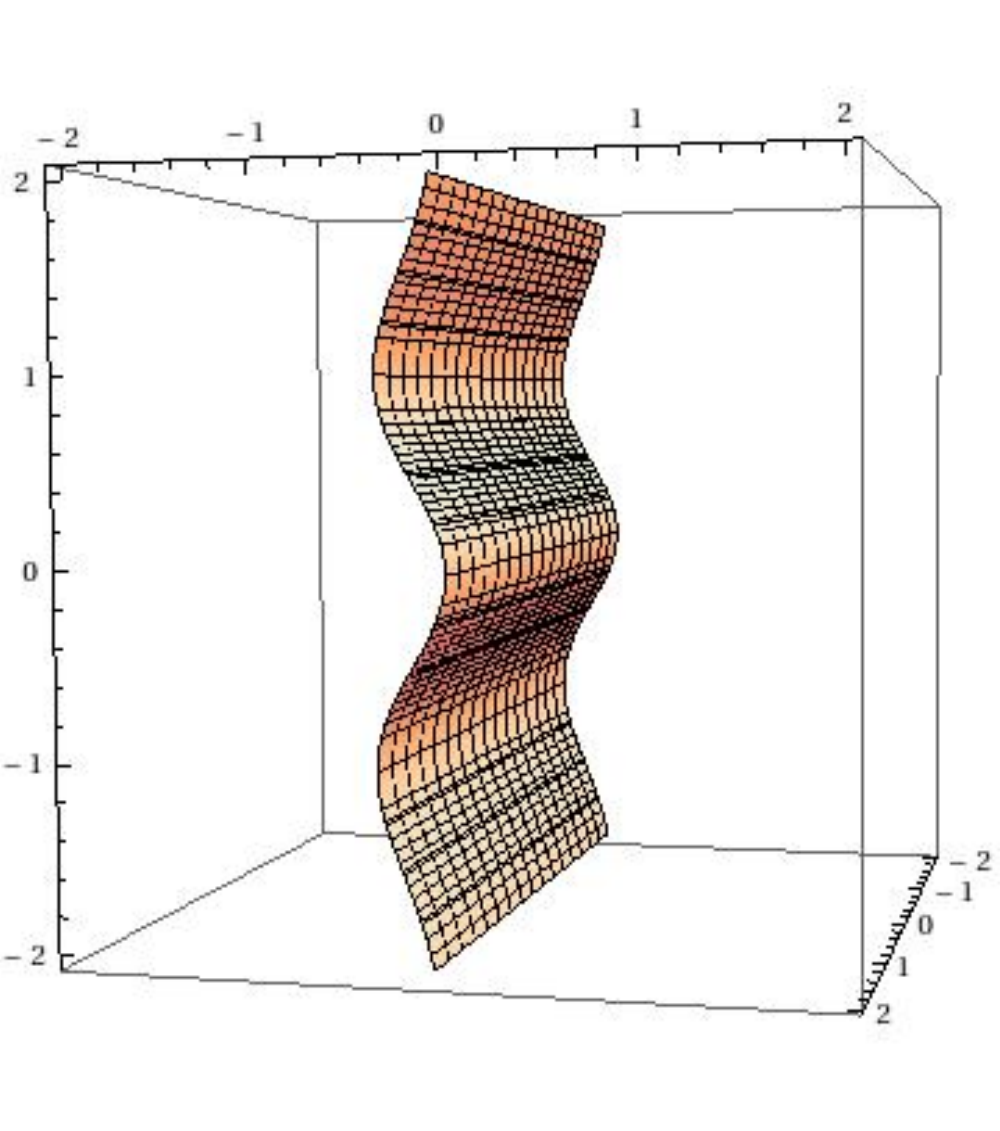}}%
    \put(0.97,0.18){\color[rgb]{0,0,0}\makebox(0,0)[lb]{\smash{$x_3$}}}%
    \put(-0.04,0.55){\color[rgb]{0,0,0}\makebox(0,0)[lb]{\smash{$x_2$}}}%
    \put(0.43,1.03){\color[rgb]{0,0,0}\makebox(0,0)[lb]{\smash{$x_1$}}}%
  \end{picture}%
\endgroup
	\caption{Example of a surface produced with $h(t)=t^{2}e^{-t^{2}}$; $x_1$, $x_2$, $x_3$ denote the coordinates in the direction of $n$, $a\wedge n$ 
and the vector perpendicular to both $n$ and $a\wedge n$ respectively. The vectors $n=(1,0,0)^T$ and $a=(0,0,1)^T$ have been chosen arbitrarily for 
simplicity and there is no smallness assumption imposed on $\Vert\dot{h}\Vert_{\infty}$ so that the curvature of the surface is clearly seen.}
 \label{fig:curvedmathematica}
\end{figure}
\end{example}

\begin{remark}
It is worth noting that if the determinant of the martensitic variants is 1, i.e.~$\Delta=1$ then, at least for martensitic transformations with cubic 
austenite, the identity matrix is an element of $K^{qc}$ (see Bhattacharya~\cite{bhattacharya1992self}) and the underlying martensitic 
microstructure can form any non-planar interface with the identity as compatibility is trivial.
\end{remark}

In proving Lemma~\ref{lemma:curvedinterface}, we assumed that there exists a relative interior point of $K^{qc}$ which is rank-one connected to 
${\rm SO}\left(d\right)$. This is by no means a trivial assumption and we now address this point.

In the context of martensitic transformations, one is interested in compact subsets of $\mathbb{R}^{d\times d}$ of the form
\begin{equation}
K=\bigcup^{N}_{i=1}{\rm SO}\left(d\right)U_{i}\nonumber
\end{equation}
where the matrices $U_{i}$ are positive definite, symmetric with $\det U_{i}=\Delta$. For $d=2$, there is a characterization of $K^{qc}$ 
(Theorem 2.2.3, in Dolzmann~\cite{dolzmannbook}) saying that $K^{qc}=K^{(2)}$, the set of laminates of order up to 2; in fact, it is easy to 
deduce from the proof that second order laminates are contained in the interior of $K^{qc}$ relative to the determinant constraint. Rank-one connections 
between second order laminates and ${\rm SO}(2)$ indeed exist and the construction of the curved interface is possible in this case. However, the case $d=3$ 
is of greater interest and, there, the situation is entirely non-trivial. For instance, in the case of two wells
\begin{equation*}
U_{1}=\mathrm{diag}\left(\eta_{1},\eta_{2},\eta_{3}\right),\quad U_{2}=\mathrm{diag}\left(\eta_{2},\eta_{1},\eta_{3}\right),
\end{equation*}
the quasiconvex hull of the set $K={\rm SO}(3)U_{1}\cup {\rm SO}(3)U_{2}$ equals $K^{(2)}$ and consists of those matrices $F$ such that
\begin{equation}
\label{eq:qchulltwowells}
F^{T}F=\left(\begin{array}{lcr}
a & c &0\\
c & b &0\\
0&0&\eta ^{2}_{3}\end{array}\right)
\end{equation}
where $ab-c^{2}=\eta^{2}_{1}\eta^{2}_{2}$ and $a+b+2\vert c\vert\leq\eta^{2}_{1}+\eta^{2}_{2}$ (see~\cite{bj92,dolzmannbook}). Note that 
due to the determinant constraint this is a two-dimensional set which implies that $K^{qc}$ is of dimension 5 (since ${\rm SO}(3)$ has 
dimension 3). On the other hand, were the relative interior of $K^{qc}$ non-empty, it would have dimension 8. Therefore, $\mathrm{rint}K^{qc}=\emptyset$ 
and our construction cannot be applied.

Hence, for $K\subset\mathbb{R}^{3\times 3}$ we follow a different approach to prove the existence of rank-one connections between ${\rm SO}(3)$ and the 
relative interior of $K^{qc}$; our argument is based on the following lemma.

\begin{lemma}[Dolzmann-Kirchheim~\cite{dolzmannkirchheim2003}]
Let $\kappa>0$, $\kappa\neq 1$ and assume that the set $\tilde{K}\subset\left\{A\in\mathbb{R}^{3\times 3}\,:\,\det\,A=1\right\}$ is compact and that 
$\tilde{K}^{qc}$ contains a three-well configuration $\tilde{K}_{ct}$ given by
\begin{eqnarray}
\tilde{K}_{ct}&=&\bigcup^{3}_{i=1}{\rm SO}\left(3\right)\tilde{U}_{i}\quad\mbox{where}\nonumber\\
\tilde{U}_{1}&=&\mathrm{diag}\left(\kappa^{2},\frac{1}{\kappa},\frac{1}{\kappa}\right),\;\tilde{U}_{2}=\mathrm{diag}\left(\frac{1}{\kappa},\kappa^{2},
\frac{1}{\kappa}\right),\;\tilde{U}_{3}=\mathrm{diag}\left(\frac{1}{\kappa},\frac{1}{\kappa},\kappa^{2}\right).\nonumber
\label{eq:tildeKct}
\end{eqnarray}
Then there exists $\epsilon=\epsilon(\kappa)>0$ such that
\[B\left(\mathbf{1},\epsilon\right)\cap\left\{A\in\mathbb{R}^{3\times 3}\,:\,\det\,A=1\right\}\subset\tilde{K}^{qc}.\]
In particular, if $\kappa<3/2$, $\epsilon(\kappa)=(\kappa-1)^2/62$ suffices.
\label{lemma:dolzmannkirchheim2003}
\end{lemma}

\begin{remark}
The three well configuration $\tilde{K}_{ct}$ corresponds to a cubic-to-tetragonal transformation for which $\eta^2_{1}=1/\eta_2$. Also, we note that 
the assumption $\kappa<3/2$ is realistic as for shape-memory alloys the lattice parameters are typically close to $1$. Henceforth, without loss of 
generality, we assume that $\epsilon<1$.
\end{remark}

From Lemma~\ref{lemma:dolzmannkirchheim2003} we deduce the following result providing conditions on $K$ such that rank-one connections between ${\rm SO}(3)$ and 
$\mathrm{rint}\,K^{qc}$ exist:

\begin{theorem}
Let $\kappa>0$, $\kappa\neq 1$ and $\Delta>0$ be such that
\begin{equation}
\frac{\vert \Delta^{1/3}-1\vert}{\Delta^{1/3}}\sqrt{\Delta^{4/3} + 2\Delta + \Delta^{2/3} + 2}<\epsilon(\kappa),
\label{eq:deltaepsilon}
\end{equation}
where $\epsilon(\kappa)\in (0,1)$ is such that $B\left(\mathbf{1},\epsilon\right)\cap
\left\{A\in\mathbb{R}^{3\times 3}:\det A=1\right\}\subset\tilde{K}^{qc}$ (see Lemma~\textnormal{\ref{lemma:dolzmannkirchheim2003}}). Further, assume that $K\subset\left\{A\in\mathbb{R}^{3\times 3}:
\det A=\Delta\right\}$ is compact and that $K^{qc}$ contains a three-well configuration $K_{ct}$ given by
\begin{equation}
K_{ct}=\bigcup^{3}_{i=1}{\rm SO}\left(3\right)U_{i},
\label{eq:Kct}
\end{equation}
where
\begin{equation*}
U_{1}=\mathrm{diag}\left(\eta_2,\eta_1,\eta_1\right),\;
U_{2}=\mathrm{diag}\left(\eta_1,\eta_2,\eta_1\right),\;
U_{3}=\mathrm{diag}\left(\eta_1,\eta_1,\eta_2\right)
\end{equation*}
and $\eta_1=\Delta^{1/3}/\kappa$, $\eta_2=\Delta^{1/3}\kappa^2$. Then there exist $a,\,n\in\mathbb{R}^{3}$ such that $\mathbf{1}+a\otimes n\in
\mathrm{rint}\,K^{qc}$.
\label{lemma:interiorpointexistence}
\end{theorem}

In proving the above theorem, we use two very simple observations which we now prove in the form of a lemma.
\begin{lemma}
In the notation of Lemma~\textnormal{\ref{lemma:dolzmannkirchheim2003}} and Theorem~\textnormal{\ref{lemma:interiorpointexistence}} the following hold:
\begin{itemize}
\item[(i)] $K^{qc}_{ct}=\Delta^{1/3}\tilde{K}^{qc}_{ct}$;
\item[(ii)] $\mathrm{rint}\,K^{qc}_{ct}=\Delta^{1/3}\mathrm{rint}\,\tilde{K}^{qc}_{ct}$; in particular, writing $F=\Delta^{1/3}\tilde{F}$,
\begin{equation*}
B(\tilde{F},\epsilon)\cap\left\{A\in\mathbb{R}^{3\times 3}:\det A=1\right\}\subset\tilde{K}^{qc}_{ct}
\end{equation*}
if and only if
\begin{equation*}
B(F,\Delta^{1/3}\epsilon)\cap\left\{A\in\mathbb{R}^{3\times 3}:\det A=\Delta\right\}\subset K^{qc}_{ct}.
\end{equation*}
\end{itemize}
\label{lemma:deltaK}
\end{lemma}
\begin{proof} (i) For notational convenience let $K=K_{ct}=\bigcup^{3}_{i=1}{\rm SO}\left(3\right)U_{i}$ and similarly, 
$\tilde{K}=\tilde{K}_{ct}=\bigcup^{3}_{i=1}{\rm SO}\left(3\right)\tilde{U}_{i}$. Note that $\Delta^{1/3}\tilde{K}=K$. To show that 
$\Delta^{1/3}\tilde{K}^{qc}\subset\,K^{qc}$, let $\tilde{F}\in\tilde{K}^{qc}$. By the characterization of the elements of $\tilde K^{qc}$ as weak$\ast$ 
limits (see Section~\ref{sec:model}), there exists a sequence $\tilde{y}^j$ uniformly bounded in $W^{1,\infty}(\Omega,\mathbb{R}^3)$, for some bounded 
domain $\Omega\subset\mathbb{R}^3$, such that
$$\mathrm{dist}(D\tilde{y}^j,\tilde{K})\rightarrow0\:\:\:\mbox{in measure},$$
$$D\tilde{y}^j\stackrel{\ast}{\rightharpoonup}\tilde{F}\:\:\:\mbox{in $L^{\infty}(\Omega,\mathbb{R}^{3\times 3})$}.$$
Define $y^{j}(x)=\tilde{y}^j(\Delta^{1/3} x)$; by our assumptions on $\tilde{y}^j$, $y^j$ is uniformly bounded in 
$W^{1,\infty}(\Delta^{-1/3}\Omega,\mathbb{R}^3)$ and the gradients $Dy^j(x)=\Delta^{1/3} D\tilde{y}^j(\Delta^{1/3} x)$ satisfy
$$\mathrm{dist}(Dy^j,\Delta^{1/3}\tilde{K})\rightarrow0\:\:\:\mbox{in measure},$$
$$Dy^j\stackrel{\ast}{\rightharpoonup}\Delta^{1/3}\tilde{F}\:\:\:\mbox{in $L^{\infty}(\Delta^{-1/3}\Omega,\mathbb{R}^{3\times 3})$}.$$
But $\Delta^{1/3}\tilde{K}=K$ implying that $\Delta^{1/3}\tilde{F}\in K^{qc}$.
That $K^{qc}\subset \Delta^{1/3}\tilde K^{qc}$ is proved similarly.

(ii) Let $\tilde{F}\in\mathrm{rint}\tilde{K}^{qc}$; then there exists $\epsilon >0$ such that
\[B(\tilde{F},\epsilon)\cap\lbrace\,A\in\mathbb{R}^{3\times 3}:\det\,A=1\rbrace\subset\tilde{K}^{qc}.\]
Let $F=\Delta^{1/3}\tilde{F}$; then $F\in\,K^{qc}$ by part (i). We wish to conclude that $F\in\mathrm{rint}\,K^{qc}$ and, in particular, that 
$B(F,\Delta^{1/3}\epsilon)\cap\lbrace\,A:\det\,A=\Delta\rbrace\subset\,K^{qc}$. It is easy to see that for any 
$G\in\,B(F,\Delta^{1/3}\epsilon)\cap\lbrace\,A:\det\,A=\Delta\rbrace$, $\vert\Delta^{-1/3}G-\tilde{F}\vert<\epsilon$ and $\det\Delta^{-1/3}G=1$. 
Therefore,
\[\Delta^{-1/3}G\in B(\tilde{F},\epsilon)\cap\lbrace\,A\in\mathbb{R}^{3\times 3}:\det A=1\rbrace\subset\tilde{K}^{qc}\]
so that $\Delta^{-1/3}G\in\tilde{K}^{qc}=\Delta^{-1/3}K^{qc}$ from part (i) and $G\in\,K^{qc}$. The reverse implication is proved similarly.
\end{proof}

\begin{proof}[Proof of Theorem~\ref{lemma:interiorpointexistence}]
We prove the result for the case $K=K_{ct}$ and the general statement then follows. 
From Lemma~\ref{lemma:dolzmannkirchheim2003} and Lemma~\ref{lemma:deltaK}, we know that the relative interior of $K^{qc}$ is non-empty and, in 
particular, $\Delta^{1/3}\mathbf{1}\in\mathrm{rint}\,K^{qc}$. The idea behind the proof is then rather natural: if we choose $\Delta>0$ close enough to 
$1$, we can surely find a rank-one direction, say $a\otimes n$, such that the line
\[\mathbf{1}+ta\otimes n,\,\,t\in\mathbb{R}\]
intersects the relative neighbourhood of $\Delta^{1/3}\mathbf{1}$ lying in $K^{qc}$, that is the set $$B\left(\Delta^{1/3}\mathbf{1},\Delta^{1/3}\epsilon\right)\cap\left\{A\in\mathbb{R}^{3\times 3}:\det A=\Delta\right\}.$$ Then, the point of intersection, say $F$, will itself be in 
$\mathrm{rint}\,K^{qc}$. Choose any vector $n\in\mathbb{R}^{3}$, $\vert n\vert=1$ and let $a=(\Delta - 1)n$. We claim that $F=\mathbf{1}+a\otimes n$ is the desired point. Trivially 
\[
\det F=1+a\cdot n = \Delta
\]
and it remains to show that
\begin{equation}
\vert F-\Delta^{1/3}\mathbf{1}\vert <\Delta^{1/3}\epsilon.\nonumber
\end{equation}
But, with $a=(\Delta-1)n$, we obtain that
\begin{eqnarray}
\vert F-\Delta^{1/3}\mathbf{1}\vert ^{2}&=&\vert(1-\Delta^{1/3})\mathbf{1}+a\otimes n\vert ^{2}\nonumber\\
&=&3(1-\Delta^{1/3})^{2}+2(1-\Delta^{1/3})a\cdot n+\vert a\vert ^{2}\nonumber\\
&=&(1-\Delta^{1/3})^{2}(\Delta^{4/3}+2\Delta + \Delta^{2/3}+2)\nonumber\\
&<&\Delta^{2/3}\epsilon^{2},\nonumber
\end{eqnarray}
where the last inequality follows from (\ref{eq:deltaepsilon}). This completes the proof.
\end{proof}

Combining the above result with our construction of the non-planar interface in Lemma~\ref{lemma:curvedinterface} we deduce that under the hypotheses of 
Theorem~\ref{lemma:interiorpointexistence}, we can construct a stress-free curved austenite-martensite interface for a set of martensitic wells 
containing the three well configuration $K_{ct}$. As remarked already, the configuration $K_{ct}$ in (\ref{eq:Kct}) corresponds to a cubic-to-tetragonal 
transition; nevertheless, such interfaces have not so far been observed in materials with such high symmetry in the martensitic phase. Thus, proving the 
existence of this relative interior point for transformations with lower martensitic symmetry is desirable. Indeed, through 
Theorem~\ref{lemma:interiorpointexistence}, we can prove the existence of rank-one connections between ${\rm SO}(3)$ and $\mathrm{rint}\,K^{qc}$ for any 
transformation with cubic austenite (with special lattice parameters) through the machinery used by Bhattacharya in~\cite{bhattacharya1992self}; in 
particular, this includes the cubic-to-orthorhombic transition undergone by Seiner's CuAlNi specimen

To be more specific, let $\mathcal{P}^{24}\subset {\rm SO}(3)$ denote the symmetry group of the cube; namely, writing $e_1$, $e_2$, $e_3$ for the standard 
basis vectors in $\mathbb{R}^{3}$ and $R[\theta, e]$ for the rotation by angle $\theta$ about the vector $e\in\mathbb{R}^3$, $\mathcal{P}^{24}$ consists 
of the following rotations:\vspace{0.2cm}

\begin{centering}
$\mathbf{1}$,\\
\quad\\
$R[\pm90^{\circ},e_1],R[\pm90^{\circ},e_2],R[\pm90^{\circ},e_3],$\\
\quad\\
$R[\pm120^{\circ},e_1+e_2+e_3], R[\pm120^{\circ},-e_1+e_2+e_3],$\\
\quad\\
$R[\pm120^{\circ},e_1-e_2+e_3],R[\pm120^{\circ},e_1+e_2-e_3],$\\
\quad\\
$R[180^{\circ},e_1],R[180^{\circ},e_2],R[180^{\circ},e_3],$\\
\quad\\
$R[180^{\circ},e_1\pm e_2],R[180^{\circ},e_2\pm e_3],R[180^{\circ},e_3\pm e_1]$.\\
\end{centering}\vspace{0.1cm}

\begin{lemma}
Let $U\in\mathbb{R}^{3\times3}$ be positive definite, symmetric with $\det U=\Delta$ and let
\[K:=\bigcup_{R\in\mathcal{P}^{24}}{\rm SO}(3)R^TUR.\]
Then, the set $K^{qc}$ contains the three-well configuration $K_{ct}$ given by
\begin{equation}
K_{ct}=\bigcup^{3}_{i=1}{\rm SO}\left(3\right)V_{i}
\label{eq:Kctmnx}
\end{equation}
with
\begin{equation*}
V_{1}=\mathrm{diag}\left(\mu,\sqrt{\nu\xi},\sqrt{\nu\xi}\right),\;V_{2}=\mathrm{diag}\left(\sqrt{\nu\xi},\mu,\sqrt{\nu\xi}\right),
\;V_{3}=\mathrm{diag}\left(\sqrt{\nu\xi},\sqrt{\nu\xi},\mu\right),\nonumber
\end{equation*}
and $\mu$, $\nu$, $\xi$ taking distinct values in the set 
\[\left\{\frac{\Delta}{\sqrt{(\mathrm{cof}\,U^2)_{jj}}},\sqrt{\frac{(\mathrm{cof}\,U^2)_{jj}}{(U^2)_{kk}}},\sqrt{(U^2)_{kk}}\right\},\]
where $j,k=1,2,3$, $j\neq k$, and for $A\in\mathbb{R}^{3\times3}$, $(\mathrm{cof}\,A)_{jj}$ denotes the $(jj)$-component of the cofactor matrix of $A$.
\label{lemma:configcubicaustenite}
\end{lemma}

\begin{proof}
For much of the proof, we follow Bhattacharya~\cite{bhattacharya1992self}; nevertheless, to retain completeness, we repeat all necessary arguments. 
Let $R=R[180^{\circ},e_1]=-{\bf 1}+2e_1\otimes e_1\in\mathcal{P}^{24}$; by Mallard's law (see Proposition 2.2 in~\cite{bhattacharya1992self}) there exist $Q\in {\rm SO}(3)$, 
$a\in\mathbb{R}^3$ such that $QUR-U=a\otimes e_1$. In particular, for all $\lambda\in[0,1]$,
\[F_{\lambda}=\lambda QUR + (1-\lambda)U = U + \lambda a\otimes e_1\in K^{qc}.\]
Set $C_{\lambda}=F^{T}_{\lambda}F_{\lambda}$; then,
$$C_{\lambda}=U^2+\lambda[Ua\otimes e_1+e_1\otimes Ua]+\lambda^2\vert a\vert^2e_1\otimes e_1,$$
$$C_0=\left(\begin{array}{rrr}
(C_0)_{11}&(C_0)_{12}&(C_0)_{13}\\
(C_0)_{12}&(C_0)_{22}&(C_0)_{23}\\
(C_0)_{13}&(C_0)_{23}&(C_0)_{33}             
\end{array}\right),\quad C_1=RC_0R=\left(\begin{array}{rrr}
(C_0)_{11}&-(C_0)_{12}&-(C_0)_{13}\\
-(C_0)_{12}&(C_0)_{22}&(C_0)_{23}\\
-(C_0)_{13}&(C_0)_{23}&(C_0)_{33}             
\end{array}\right),$$
with $C_0=U^2$. Note that:
\begin{itemize}
 \item[(1)] if $ij\neq11$ then $(C_\lambda)_{ij}=C_\lambda e_j\cdot e_i=(C_0)_{ij}+\lambda a(i,j)$ is affine in $\lambda$;
 \item[(2)] if $ij=12,21,31,13$ then $(C_1)_{ij}=-(C_0)_{ij}$;
 \item[(3)] if $i\neq1,\,j\neq1$ then $(C_1)_{ij}=(C_0)_{ij}$.
\end{itemize}
By (1) and (2), evaluating for $\lambda=1$, we infer that $a(i,j)=-2(C_0)_{ij}$ for $ij=12$, $21$, $31$, $13$ and hence $(C_{1/2})_{ij}=0$; similarly, by (1) and 
(3), we find that $(C_{1/2})_{ij}=(C_0)_{ij}$ for $i\neq1$, $j\neq1$. This implies that
$$C_{1/2}=\left(\begin{array}{ccc}
(C_{1/2})_{11}&0&0\\
0&(C_0)_{22}&(C_0)_{23}\\
0&(C_0)_{23}&(C_0)_{33}           
\end{array}\right).$$
But $F_{1/2}\in K^{qc}$ implies $\det F_{1/2}=\Delta$ and thus
\begin{equation}
 \label{eq:C1/211}
(C_{1/2})_{11}=\frac{\Delta^2}{(\mathrm{cof}\,U^2)_{11}}.
\end{equation}
Now let $R=R[180^{\circ},e_2]\in\mathcal{P}^{24}$. We can find $Q\in {\rm SO}(3)$, $a\in\mathbb{R}^3$ such that $QF_{1/2}R-F_{1/2}=a\otimes e_2$ and, for all 
$\lambda\in[0,1]$,
\[G_{\lambda}=\lambda QF_{1/2}R + (1-\lambda)F_{1/2} = F_{1/2} + \lambda a\otimes e_2\in K^{qc}.\]
Setting $D_{\lambda}=G^{T}_{\lambda}G_{\lambda}$ and repeating the above process, we deduce that
$$D_{1/2}=\left(\begin{array}{ccc}
\frac{\Delta^2}{(\mathrm{cof}\,U^2)_{11}}&0&0\\
0&\frac{(\mathrm{cof}\,U^2)_{11}}{(U^2)_{33}}&0\\
0&0&(U^2)_{33}           
\end{array}\right),$$
where, we have used (\ref{eq:C1/211}) and the determinant constraint to calculate $(D_{1/2})_{22}$. It follows that $\sqrt{D_{1/2}}\in K^{qc}$. However, 
for any $Q\in {\rm SO}(3)$, $R\in\mathcal{P}^{24}$, $Q\sqrt{D_{1/2}}R\in K^{qc}$ (see \cite{bhattacharya1992self}) and using $Q=R$ with 
$R=R[180^{\circ}, e_2+e_3]$, $R=R[180^{\circ}, e_1+e_3]$ and $R=R[180^{\circ}, e_1+e_2]$, we find that
\[\bigcup^{6}_{i=1}{\rm SO}(3)\tilde{V}_{i}\subset K^{qc},\,\,\mbox{where}\]
$$\tilde{V}_{1}=\sqrt{D_{1/2}}=\left(\begin{array}{ccc}
\frac{\Delta}{\sqrt{(\mathrm{cof}\,U^2)_{11}}}&0&0\\
0&\sqrt{\frac{(\mathrm{cof}\,U^2)_{11}}{(U^2)_{33}}}&0\\
0&0&\sqrt{(U^2)_{33}}          
\end{array}\right)=:\mathrm{diag}(\alpha,\beta,\gamma)$$
and $\tilde{V}_i$, $i=2,\ldots,6$ are given by permuting the components of $\tilde{V}_1$ on the diagonal. Next consider, for example, the matrices  
$\tilde{V}_1=\mathrm{diag}(\alpha,\beta,\gamma)$, $\tilde{V}_2=\mathrm{diag}(\alpha,\gamma,\beta)$;
Then ${\rm SO}(3)\tilde{V}_1\cup {\rm SO}(3)\tilde{V}_2\subset K^{qc}$ and by the two-well problem - see the comment at (\ref{eq:qchulltwowells}) - we infer that
\[\mathrm{diag}(\alpha,\sqrt{\beta\gamma},\sqrt{\beta\gamma})\in K^{qc}.\]
In particular, using the $180^{\circ}$ rotations about the diagonals as above, we see that the three-well configuration
\[{\rm SO}(3)\mathrm{diag}(\alpha,\sqrt{\beta\gamma},\sqrt{\beta\gamma})\cup {\rm SO}(3)\mathrm{diag}(\sqrt{\beta\gamma},\alpha,\sqrt{\beta\gamma})\cup {\rm SO}(3)
\mathrm{diag}(\sqrt{\beta\gamma},\sqrt{\beta\gamma},\alpha)\]
belongs to $K^{qc}$. It is easy to see that, by considering the other possible pairs of $\tilde{V}_i$, we can interchange the roles of $\alpha$, $\beta$ 
and $\gamma$, to get another two three-well configurations belonging to $K^{qc}$. We have now obtained our result with $j=1$, $k=2$ in the expressions 
for $\mu$, $\nu$, $\xi$. This is because we reached the diagonal element $\tilde{V}_1\in K^{qc}$ by first applying $R=R[180^{\circ},e_1]$ and then 
$R=R[180^{\circ},e_2]$. Alternatively, one can do the diagonalization by first applying $R=R[180^{\circ},e_1]$ and then $R=R[180^{\circ},e_3]$ or any 
of the other four possibilities and the result follows.
\end{proof}

Theorem~\ref{lemma:interiorpointexistence} now allows us to deduce the existence of rank-one connections between ${\rm SO}(3)$ and $\mathrm{rint}\,K^{qc}$ for 
any transition with cubic austenite as in Lemma~\ref{lemma:configcubicaustenite}:

\begin{corollary}
Let $U\in\mathbb{R}^{3\times3}$ be positive definite, symmetric with $\det U=\Delta$ and satisfy
\begin{equation}
\label{eq:corollary}
\frac{\vert\Delta^{1/3}-1\vert}{\Delta^{1/3}}\sqrt{\Delta^{4/3} + 2\Delta + \Delta^{2/3} + 2}<\epsilon(\kappa),
\end{equation}
for some $\kappa\in\mathcal S(U)$, $\kappa>0$, $\kappa\neq1$, where
\begin{equation*}
\mathcal S(U):=\bigcup_{j\neq k}\left\{\frac{\Delta^{1/3}}{(\mathrm{cof}\,U^2)^{1/4}_{jj}},\frac{(\mathrm{cof}\,U^2)^{1/4}_{jj}}{(U^2)^{1/4}_{kk}\Delta^{1/6}},
\frac{(U^2)^{1/4}_{kk}}{\Delta^{1/6}}\right\}
\end{equation*}
and $0<\epsilon(\kappa)<1$ is as in Lemma~\textnormal{\ref{lemma:dolzmannkirchheim2003}}. Let
\[K:=\bigcup_{R\in\mathcal{P}^{24}}{\rm SO}(3)R^TUR.\]
Then there exist $a,\,n\in\mathbb{R}^{3}$ such that $\mathbf{1}+a\otimes n\in\mathrm{rint}\,K^{qc}$.
\label{corollary:interiorpointcubicaustenite}
\end{corollary}

\begin{proof}
The proof follows immediately by Theorem~\ref{lemma:interiorpointexistence} and Lemma~\ref{lemma:configcubicaustenite}. For example, suppose that 
(\ref{eq:corollary}) holds for $\kappa=\Delta^{1/3}/(\mathrm{cof}\,U^2)^{1/4}_{jj}$ and some $j\neq k$ fixed; by 
Lemma~\ref{lemma:configcubicaustenite} the three-well configuration in (\ref{eq:Kctmnx}) is contained in $K^{qc}$ where 
$\mu=\Delta/\sqrt{(\mathrm{cof}\,U^2)_{jj}}$, $\nu=\sqrt{(\mathrm{cof}\,U^2)_{jj}/(U^2)_{kk}}$ and $\xi=\sqrt{(U^2)_{kk}}$. But then
\[\mu=\Delta^{1/3}\kappa^2,\,\sqrt{\nu\xi}=\frac{\Delta^{1/3}}{\kappa}\]
and Theorem~\ref{lemma:interiorpointexistence} applies to give the result.
\end{proof}

Unfortunately for the CuAlNi specimen of Seiner the value of $\kappa$ given in Theorem~\ref{lemma:interiorpointexistence} is not large enough for Corollary~\ref{corollary:interiorpointcubicaustenite} to apply and establish the existence of a non-planar interface. To see this note that CuAlNi undergoes a cubic-to-orthorhombic transition and the corresponding transformation strain $U$ is given by
\[
U=\left(\begin{array}{ccc}
\beta & 0 & 0\\ 0 & \frac{\alpha+\gamma}{2} & \frac{\alpha-\gamma}{2}\\ 0 & \frac{\alpha-\gamma}{2} & \frac{\alpha+\gamma}{2}
\end{array}\right).
\]
In accordance with~\cite{50}, let $\alpha = 1.06372$, $\beta = 0.91542$ and 
$\gamma = 1.02368$ be the lattice parameters; then $\Delta^{1/3}=(\alpha\beta\gamma)^{1/3}=0.998935$ and it turns out that the value of $\kappa$ in the set ${\mathcal S}(U)$ that maximizes $(\kappa-1)^2$ 
is $\kappa^*=\beta^{1/3}(\alpha\gamma)^{-1/6}=0.957286$. In particular, $\kappa^*<3/2$ and we may take $\epsilon=(\kappa^*-1)^2/62=2.94277\times 10^{-5}$. But then
$$2.60824\times 10^{-3}=\frac{\vert \Delta^{1/3}-1\vert}{\Delta^{1/3}}\sqrt{\Delta^{4/3}+2\Delta+\Delta^{2/3} + 2}>\epsilon=2.94277\times 10^{-5},$$
so that (\ref{eq:corollary}) does not hold. For other materials $\Delta$ might be closer to 1 so that Corollary~\ref{corollary:interiorpointcubicaustenite} applies.

\section{Concluding remarks}

Our method of constructing a non-planar interface for a set of martensitic wells $K$ depends on the existence of rank-one connections 
between ${\rm SO}(3)$ and $\mathrm{rint}\,K^{qc}$. Since $K^{\rm qc}$ is not known for more than two martensitic energy-wells, establishing the existence of such rank-one connections is a difficult problem. Above, we gave conditions for the existence of such rank-one connections for any transition with cubic austenite, but the method of doing so relied on finding an embedded cubic-to-tetragonal configuration and did not exploit the potentially rich structure of $K^{qc}$. The resulting restriction on the determinant seems much too strong for typical values of the lattice parameters. Possibly, via enlarging the 
neighbourhood of $\mathbf{1}\in\mathbb{R}^{3\times 3}$ from Lemma~\ref{lemma:dolzmannkirchheim2003}, we might be able to predict non-planar 
interfaces for existing alloys or even Seiner's specimen. There could also be entirely new ways of constructing non-planar interfaces which come with less stringent assumptions on the lattice parameters. This is ultimately a problem on quasiconvex hulls and appropriate jump 
conditions, both of which pose deep and interesting questions. Moreover, we note that the analysis does not reveal much about the microstructure corresponding to the 
relative interior point. Based on \cite{dolzmannkirchheim2003}, this interior point must be a (potentially high order) laminate but we cannot be more explicit.

Lastly, we mention that R.~D.~James and others~\cite{cui06,jameszhangway} have extensively investigated the case when the middle eigenvalue $\lambda_2(U_i)$ of the martensitic variants equals 1 and the so-called \textit{cofactors condition} (see e.g.~\cite{jameszhangway}) holds, both theoretically, and experimentally by appropriately `tuning' the lattice parameters of alloys. This work has established a strong connection between these conditions and low thermal hysteresis. Under the cofactor condition, one is able to theoretically construct non-planar interfaces with a pure phase of austenite, 
without the need for a boundary layer; however, this is restricted to this special case and the fact that martensitic twins are 
directly compatible with the austenite~\cite{james_curved}.

\section*{Acknowledgement}

The research of both authors was supported by the EPSRC Science and Innovation award to the Oxford Centre for Nonlinear PDE (EP/E035027/1) and the European Research Council under the European Union's Seventh Framework Programme (FP7/2007-2013) / ERC grant agreement ${\rm n^o}$ 291053. The research of JMB was also supported by a Royal Society Wolfson Research Merit Award. We would like to thank Bernd Kirchheim and Hanu\v s Seiner for very useful discussions.

\end{document}